\pdfoutput=1
\documentclass[11pt,a4paper]{amsart}
\usepackage[margin=1in]{geometry}
\usepackage[T1]{fontenc}
\usepackage{textcase}
\usepackage[dvipsnames]{xcolor}
\usepackage{microtype}
\usepackage{fnpct}

\usepackage{amsmath}
\usepackage{amssymb}
\usepackage{mathtools}
\usepackage[mathscr]{eucal}
\renewcommand{\injlim}{\varinjlim}

\usepackage[inline]{enumitem}
\usepackage{booktabs}
\usepackage[pdfusetitle,colorlinks]{hyperref}
\hypersetup{bookmarksdepth=2,pdfencoding=unicode,allcolors=MidnightBlue}

\usepackage{zref-clever}
\zcsetup{abbrev=false,cap=true,nameinlink=false,sort=false,lang=english}
\newcommand{\cref}[1]{\zcref{#1}}
\newcommand{\Cref}[1]{\zcref[S]{#1}}
\zcsetup{pairsep={ and~},lastsep={, and~}}
\zcRefTypeSetup{equation}{Name-sg=,Name-pl=,refbounds={(,,,)}}
\zcRefTypeSetup{item}{Name-sg=,Name-pl=,refbounds={(,,,)}}

\NewDocumentCommand{\newzctheorem}{momo}{\IfValueTF{#4}
  {\newtheorem{#1}{#3}[#4]}
  {\IfValueTF{#2}
    {\AddToHook{env/#1/begin}{\zcsetup{countertype={#2=#1}}}\newtheorem{#1}[#2]{#3}}
    {\newtheorem{#1}{#3}}}}

\theoremstyle{plain}
\newzctheorem{theorem}[equation]{Theorem}
\newzctheorem{lemma}[equation]{Lemma}
\theoremstyle{definition}
\newzctheorem{definition}[equation]{Definition}
\newzctheorem{construction}[equation]{Construction}
\zcRefTypeSetup{construction}{Name-sg=Construction,Name-pl=Constructions}
\newzctheorem{example}[equation]{Example}
\newzctheorem{question}[equation]{Question}
\zcRefTypeSetup{question}{Name-sg=Question,Name-pl=Questions}
\theoremstyle{remark}
\newzctheorem{remark}[equation]{Remark}

\newcommand{\ann}{\operatorname{ann}}

\newcommand{\id}{\operatorname{id}}
\newcommand{\Spec}{\operatorname{Spec}}
\newcommand{\Spv}{\operatorname{Spv}}
\newcommand{\op}{\operatorname{op}}

\newcommand{\Tot}{\operatorname{Tot}}

\newcommand{\spec}{\textnormal{spec}}

\newcommand{\X}{\mathord{-}}
\newcommand{\Cat}[1]{\mathrm{#1}}

\title{Inaccessibility of the flat topology}
\author{Ko Aoki}
\address{Department of Mathematical Sciences,
  University of Copenhagen, Denmark
}
\email{aoki@math.ku.dk}
\author{Robert Burklund}
\address{Department of Mathematical Sciences,
  University of Copenhagen, Denmark
}
\email{rb@math.ku.dk}
\date{\today}

\begin{document}

\begin{abstract}
  We prove that
  singleton covers in the flat topology
  on affine schemes
  are not closed under \(\kappa\)-cofiltered limits
  for any regular cardinal~\(\kappa\).
  Therefore,
  for every accessible flat sheaf
  there exists a strictly finer topology
  for which it is still a sheaf.
  The flat topology thus contrasts with
  other big topologies, such as
  the arc and pure topologies.
\end{abstract}

\maketitle
\setcounter{tocdepth}{1}

Every accessible presheaf~\(F\) valued in an accessible \(\infty\)-category
on the category of affine schemes \(\Cat{Aff}\)
admits a finest finitary\footnote{This finitarity restriction is essential for \cref{xwizpy},
  as the canonical topology is not finitary.
} topology for which it is a sheaf,
which we denote by~\(\tau_{F}\).

\begin{example}[Olivier]\label{xwizpy}
  Let $F \coloneqq \id\colon \Cat{Aff}^{\op}\to\Cat{Aff}^{\op}$. 
  It follows from
  what was stated in~\cite{Olivier70}
  (and proven in~\cite{JoyalTierney84,Mesablishvili00,AndreFiorot22})
  that \(\tau_{F}\) is the pure topology.
  A map \(\Spec B\to\Spec A\) is a pure cover
  if and only if \(M\to M\otimes_{A}B\)
  is injective for any \(A\)-module~\(M\).
\end{example}

\begin{example}[Bhatt–Mathew, Rydh]\label{xyr7pm}
Let $F \coloneqq \Spec\colon \Cat{Aff}^{\op}\to(\Cat{Top}^{\spec})^{\op}$,
  where the target is the category of spectral topological spaces
  and quasicompact maps.
  Then \(\tau_{F}\) is the arc~topology of~\cite{BhattMathew21}.
\end{example}

\begin{example}\label{x7so2c}
  Let \(F\coloneq\Spv\colon\Cat{Aff}^{\op}\to(\Cat{Top}^{\spec})^{\op}\).
  Then \(\tau_{F}\) is the v~topology of~\cite{Rydh10,BhattScholze17}.
\end{example}

These topologies are important,
but the flat topology remains the most widely used.

\begin{question}\label{xu57pv}
  Is there an accessible presheaf~\(F\)
  such that \(\tau_{F}\) is the flat topology?
\end{question}

We call
a ring map \(A\to B\)
an \emph{fpqc cover}\footnote{\emph{Flat cover} is also a valid term;
  we avoid it in the body, since an fpqc cover need not be flat.
}
if \(\Spec B\to\Spec A\) is a cover in the flat topology;
i.e.,
there exists \(B\to C\)
with the composite \(A\to C\) faithfully flat.

In answering \cref{xu57pv}, we observe that
for accessible $F$ the singleton $\tau_F$-covers of rings
are closed under $\kappa$-filtered colimits for $\kappa\gg0$,
while the fpqc covers are \emph{not}
closed under $\kappa$-filtered colimits for any $\kappa$.

\begin{lemma}
  Let \(\kappa\) be an uncountable regular cardinal.
  If $F$ is a $\kappa$-accessible sheaf,
  then a $\kappa$-filtered colimit
  of singleton $\tau_F$-covers of rings is itself a $\tau_F$-cover.
\end{lemma}

\begin{proof}
  Let $A_i \to B_i$ be a $\kappa$-filtered diagram of $\tau_F$-covers of rings
  with colimit $A \to B$.
  Using our uncountability hypothesis,
  for any $A \to A'$, we have
  \begin{equation*}
    F(A')
    \simeq \injlim_{i} \Tot F\Bigl(A' \otimes_{A_i} B_i^{\otimes_{A_i} \bullet+1}\Bigr)
    \simeq \Tot \injlim_{i} F\Bigl(A' \otimes_{A_i} B_i^{\otimes_{A_i} \bullet+1}\Bigr)
    \simeq \Tot F\bigl(A' \otimes_A B^{\otimes_{A} \bullet+1}\bigr).
    \qedhere
  \end{equation*}
\end{proof}

\begin{theorem}\label{x6cp88}
  For any regular cardinal~\(\kappa\),
  there is a \(\kappa\)-filtered family
  of fpqc covers of rings
  whose colimit is not an fpqc cover.
\end{theorem}

\begin{proof}
    Let \(\kappa\) be a regular cardinal
    and \(G\) the \(\kappa\)-Fréchet filter on~\(\kappa\);
    i.e., \(J\in G\) if and only if
    \(\lvert\kappa\setminus J\rvert<\kappa\).
    Regularity implies \(G\) is \(\kappa\)-cofiltered.
    Let \(k\) be a field, let
    \begin{align*}
      A &\coloneqq \frac
      {k[x, y_{\alpha}, z_{\alpha} \mid \alpha <\kappa]}
      {
        \langle xy_{\alpha} \mid \alpha <\kappa\rangle
        +
        \langle y_{\alpha}z_{\beta} \mid \alpha < \beta <\kappa\rangle
        +
        \langle y_{\alpha}z_{\beta} - z_{\beta}^{2} \mid \beta\leq\alpha<\kappa\rangle
      },&
      B &\coloneqq \injlim_{J \in G^{\op}} A^{J},
    \end{align*}
and let \(y \coloneqq (y_{\alpha})_{\alpha < \kappa} \in B\).
  
  The diagonal map \( \Delta \colon A\to B \)
  is a \(\kappa\)-filtered colimit of split injections.
  We prove that \(\Delta\) is not an fpqc cover.
  More specifically, 
  given $f\colon B \to C$ 
  with $h \coloneqq f \circ \Delta$ flat
  we show that $h$ is not injective,
  hence not faithfully flat.
  We proceed in the following steps:
  \begin{enumerate}
    \item\label{i:sq}
      \(z_{\alpha}^{2} \neq 0\)
      for any~\(\alpha < \kappa\).
    \item\label{i:ann}
      \(\ann_{A}(x) = \langle y_{\alpha} \mid \alpha<\kappa\rangle\).
    \item\label{i:flat}
      $h(z_\alpha^2) = 0$
      for some $\alpha < \kappa$.
  \end{enumerate}
  For~\cref{i:sq},
  consider the map $g\colon A \to k[t]$ given by
  \(g(x) = 0\),
  \(g(y_{\gamma}) = 0\) for \(\gamma<\alpha\),
  \(g(y_{\gamma}) = t\) for \(\gamma\geq\alpha\),
  \(g(z_{\alpha}) = t\) and
  \(g(z_{\gamma}) = 0\) for \(\gamma\neq\alpha\).
  For~\cref{i:ann}, one inclusion is clear.
  The other follows by observing that
  \(x\) is regular in the quotient
  \(
  A/\langle y_{\alpha}\mid\alpha<\kappa\rangle
  \simeq
  k[x, z_{\alpha}\mid\alpha<\kappa]/
  \langle z_{\alpha}^{2}\mid\alpha<\kappa\rangle
  \).
  We then prove~\cref{i:flat}.
  From the flatness of~$h$, we deduce that
  \(\ann_{C}(h(x))\simeq\ann_{A}(x)C\).
  By \cref{i:ann} and the relation $\Delta(x)y = 0$,
  we conclude that there exists a finite \(S \subset \kappa\) and a 
  \(c \in C^S\) such that 
  \( f(y) = c \cdot(h(y_s))_{s \in S} \).
  Let \(\alpha \coloneqq \max(S\cup\{0\})+ 1 < \kappa\).
  We compute that
  \begin{equation*}
    h(z_{\alpha}^{2})
    =f(y) h(z_{\alpha})
    =( c \cdot(h(y_s))_{s \in S} ) h(z_{\alpha})
    = c \cdot(h(y_s z_\alpha))_{s \in S}
    =0.
    \qedhere
  \end{equation*}
\end{proof}

\begin{remark}
  Despite \Cref{x6cp88},
  accessible flat sheaves remain a well-behaved notion,
  as Waterhouse~\cite{Waterhouse75}
  proved that
  accessible presheaves have accessible flat sheafifications. 
  This relies on the fact that
  faithfully flat ring maps are closed under filtered colimits
  and
  any faithfully flat map is an \(\aleph_{1}\)-filtered colimit
  of those between countable rings.
What \Cref{x6cp88} demonstrates is that
  it is essential to work with the flat \emph{pretopology} to establish this.
\end{remark}

\begin{remark}\label{xh9czb}
  Our use of $\kappa$-filtered colimits of products
  in the proof of \cref{x6cp88} is not accidental:
  Suppose \(A\to B\) is a \(\kappa\)-filtered colimit of fpqc covers \(A_{i}\to B_{i}\).
  \begin{itemize}
    \item
      Applying \(\X\otimes_{A_{i}}A\), we may assume all maps share source~\(A\).
    \item
      For each~\(i\), take a map \(B_{i}\to C_{i}\)
      such that \(A\to C_{i}\) is faithfully flat.
      Applying \(\X\otimes_{A}\bigotimes_{i}C_{i}\),
      we may assume each \(A\to B_{i}\) admits a section.
    \item
      Let~\(S_{i}\) be the set of sections of \(A\to B_{i}\).
      We may replace~\(B_{i}\) with~\(A^{S_{i}}\).
    \item
      Moreover, applying \cref{xsamv0} below,
      we may replace~\((S_{i})_{i}\) with a \(\kappa\)-complete filter.
  \end{itemize}
\end{remark}

\begin{lemma}\label{xsamv0}
  Let \(P\) be a poset.
  Any \(P\)-indexed family
  of nonempty sets \((S_{p})_{p}\)
  receives a map from a \(P\)-indexed family with injective transition maps.
\end{lemma}

\begin{proof}
  Let \(T_{p}\) be the subset of \(\prod_{q\in P}S_{q}\)
  consisting of \((s_{q})_{q}\)
  such that \(s_{q}\) is the image of \(s_{p}\) for any \(q\geq p\).
  Since \(T_{p}\simeq S_{p}\times\prod_{q\not\geq p}S_{q}\),
  it is nonempty.
  This defines a \(P\)-indexed family
  with injective transitions,
  and the projection \(T_{p}\to S_{p}\)
  yields the desired map.
\end{proof}

\subsection*{Acknowledgments}\label{ss:ack}

We would like to thank
Ishan Levy for helpful discussions.
During the course of this work
we were supported
by the Danish National Research Foundation
through the Copenhagen Center for Geometry and Topology (DNRF151).

\bibliographystyle{alpha}

\begin{thebibliography}{Wat75}

\bibitem[AF22]{AndreFiorot22}
Yves Andr\'e and Luisa Fiorot.
\newblock On the canonical, fpqc, and finite topologies on affine schemes. {T}he state of the art.
\newblock {\em Ann. Sc. Norm. Super. Pisa Cl. Sci. (5)}, 23(1):81--114, 2022.

\bibitem[BM21]{BhattMathew21}
Bhargav Bhatt and Akhil Mathew.
\newblock The arc-topology.
\newblock {\em Duke Math. J.}, 170(9):1899--1988, 2021.

\bibitem[BS17]{BhattScholze17}
Bhargav Bhatt and Peter Scholze.
\newblock Projectivity of the {W}itt vector affine {G}rassmannian.
\newblock {\em Invent. Math.}, 209(2):329--423, 2017.

\bibitem[JT84]{JoyalTierney84}
Andr\'e{} Joyal and Myles Tierney.
\newblock An extension of the {G}alois theory of {G}rothendieck.
\newblock {\em Mem. Amer. Math. Soc.}, 51(309):vii+71, 1984.

\bibitem[Mes00]{Mesablishvili00}
Bachuki Mesablishvili.
\newblock Pure morphisms of commutative rings are effective descent morphisms for modules---a new proof.
\newblock {\em Theory Appl. Categ.}, 7:No. 3, 38--42, 2000.

\bibitem[Oli70]{Olivier70}
Jean-Pierre Olivier.
\newblock Descente par morphismes purs.
\newblock {\em C. R. Acad. Sci. Paris S\'er. A-B}, 271:A821--A823, 1970.

\bibitem[Ryd10]{Rydh10}
David Rydh.
\newblock Submersions and effective descent of \'etale morphisms.
\newblock {\em Bull. Soc. Math. France}, 138(2):181--230, 2010.

\bibitem[Wat75]{Waterhouse75}
William~C. Waterhouse.
\newblock Basically bounded functors and flat sheaves.
\newblock {\em Pacific J. Math.}, 57(2):597--610, 1975.

\end{thebibliography}

\end{document}